\documentclass[a4paper, 11pt]{article}
\usepackage{amsmath,amssymb,esint,amscd,xspace,fancyhdr,color,authblk,srcltx,fontenc,bbm}
\usepackage{hyperref}
\usepackage{cite}

\newtheorem{theorem}{Theorem}[section]

\newtheorem{corollary}[theorem]{Corollary}

\newtheorem{lemma}[theorem]{Lemma}

\newenvironment{proof}[1][Proof]{\textbf{#1.} }{\hfill\rule{0.5em}{0.5em}}
{\catcode`\@=11\global\let\AddToReset=\@addtoreset
\AddToReset{equation}{section}

\AddToReset{theorem}{section}

\title{Lorentz gradient estimates for a class of elliptic $p$-Laplacian equations with a Schr{\"o}dinger term}
\author{Minh-Phuong Tran\thanks{Applied Analysis Research Group, Faculty of Mathematics and Statistics, Ton Duc Thang University, Ho Chi Minh City, Vietnam; \texttt{tranminhphuong@tdtu.edu.vn}}, Thanh-Nhan Nguyen\footnote{Corresponding author.} \thanks{Department of Mathematics, Ho Chi Minh City University of Education, Ho Chi Minh City, Vietnam; \texttt{nhannt@hcmue.edu.vn}},  Gia-Bao Nguyen\thanks{Department of Mathematics, Ho Chi Minh City University of Education, Ho Chi Minh City, Vietnam; \texttt{baong@hcmue.edu.vn}}}

\date{\today}

\begin{document}
 
\maketitle
\begin{abstract}
We prove in this paper the global Lorentz estimate in term of fractional-maximal function for gradient of weak solutions to a class of $p$-Laplace elliptic equations containing a non-negative Schr\"odinger potential which belongs to reverse H\"older classes. In particular, this class of $p$-Laplace operator includes both degenerate and non-degenerate cases. The interesting idea is to use an efficient approach based on the level-set inequality related to the distribution function in harmonic analysis.

\medskip

\medskip

\medskip

\noindent 

\medskip


\end{abstract}   
                  
\section{Introduction and main results}
\label{sec:intro}
In this article, we are interested in the global regularity for a class of quasi-linear elliptic equations, whose prototype coming from the (degenerate/non-degenerate) $p$-Laplace Schr\"odinger equations
\begin{align}\label{eq:p-Lap}
- \Delta_{p,\sigma} u + \mathbb{V} |u|^{p-2}u &=  -\mathrm{div}(|\mathbf{f}|^{p-2}\mathbf{f}) \ \text{ in } \ \Omega, \mbox{ and } \ u = \mathsf{g}  \ \text{ on } \  \partial \Omega,
\end{align}
where $\Delta_{p,\sigma} u$ denotes the $p$-Laplace operator for $p>1$, with degeneracy $\sigma\in [0,1]$. To be more precise, we consider such operator that can be written in the Euclidean setting as
\begin{align*}
\Delta_{p,\sigma} u:= \mathrm{div}\left(\left(\sigma^2+|\nabla u|^2\right)^{\frac{p-2}{2}} \nabla u\right)=\sum_{i=1}^n{\frac{\partial}{\partial x_i}\left[\left(\sigma^2+|\nabla u|^2\right)^{\frac{p-2}{2}}\frac{\partial u}{\partial x_i} \right]}.
\end{align*}
Moreover, the problem is considered in $\Omega \subset \mathbb{R}^n$-is a Reifenberg flat domain, $n \ge 2$; with fundamental data $\mathbf{f} \in L^p(\Omega; \mathbb{R}^n)$ and boundary condition $\mathsf{g} \in W^{1,p}(\Omega)$. Toward this goal, we establish the global Lorentz estimates for the gradient of weak solutions to \eqref{eq:p-Lap} in terms of \emph{fractional maximal operators} as follows
\begin{align}\label{goal}
\mathbf{M}_\alpha(\sigma^p + |\mathbf{f}|^p+|\nabla \mathsf{g}|^p + \mathbb{V}|\mathsf{g}|^p) \in L^{q,s}(\Omega) \Longrightarrow \mathbf{M}_\alpha(\sigma^p + |\nabla u|^p + \mathbb{V}|u|^p) \in L^{q,s}(\Omega),
\end{align}
where $\mathbf{M}_{\alpha}$ denotes the fractional-maximal operator of order $\alpha \in [0,n)$. In the sense of calculus of variations, it is concerned with the problem of minimizing the functional with additional potential $\mathbb{V}$
\begin{align}
\label{eq:min}
W^{1,p}(\Omega) \ni w \quad \mapsto \quad \frac{1}{p}\int_{\Omega}{\left( |\nabla w|^p + \mathbb{V}|w|^p \right) dx} - \int_\Omega{\langle |\mathbf{f}|^{p-2}\mathbf{f},\nabla w \rangle dx},
\end{align}
or in another word, the equation \eqref{eq:p-Lap} is the Euler-Lagrange equation whose solution minimizes the functional in \eqref{eq:min}; and this problem arises naturally in various areas of applied mathematics and physics. The problems with Schr\"odinger term $\mathbb{V}$ have their own difficulties and curiosities and received a special attention in the scientific community through the years. For instance, they appear frequently in quantum mechanics, non-Newtonian fluid theory, some non-linear phenomena, plasma physics, turbulent gas flow in porous media and so on (see \cite{BS1991,GV1988}). Not only the existence, the study of regularity properties or Calder\'on-Zygmund estimates for weak solutions have also played a central role in the research of nonlinear elliptic equations in recent years. Further, there has been continuous attention on these regularity properties for elliptic problems with Schr\"odinger potentials. In particular, for the case $p=2$, the equation \eqref{eq:p-Lap} reduces to a time independent Schr\"odinger equation and the $L^q$-estimates has been established by Shen in \cite{Shen} whenever the non-negative potential $\mathbb{V}$ belongs to the reverse H\"older class $\mathbb{RH}^\theta$, for some $\theta \ge \frac{n}{2}$ (a more detailed description of such class $\mathbb{RH}^\theta$ can be found below). In another interesting work \cite{Sugano}, Sugano also provided a result when $\mathbb{V}$ belongs to $\mathbb{RH}^\theta$ class which includes non-negative polynomials. As an extension of these results, Bramanti \textit{et al.} in \cite{BBHV} proved $L^q$-estimates for non-divergence form linear elliptic equations with VMO coefficients; and recently results for nonlinear divergence elliptic equations were derived by Lee and Ok in \cite{LO_schro}. Moreover, with additional assumptions on domain $\Omega$ (Reifenberg flat domain), nonlinearity $\mathbb{A}$ (satisfying a small BMO condition) and on potential $\mathbb{V}$ (in $\mathbb{RH}^\theta$ class and its appropriate Morrey norm is bounded), they stated and proved the global Calder\'on-Zygmund estimates for homogeneous problems in the same paper.

The present work is motivated by these recent contributions to the Calder\'on-Zygmund theory with a wealth of technical approaches and regularity results. It is a natural outgrowth of the previous works: the new gradient estimates via the fractional-maximal operators, that play a decisive role in our proofs, were first proposed in~\cite{PNJDE,PNmix,PNnonuniform}; and a Lorentz regularity extends the global Calder{\'o}n-Zygmund estimates established in~\cite{LO_schro} for quasi-linear elliptic Schr\"odinger equations. It is worth emphasizing that regularity results proved in our paper are also hold for a larger and more general class of elliptic equations than the one in~\eqref{eq:p-Lap}. More precisely, we deal with a nonhomogeneous Dirichlet problem for quasilinear Schr\"odinger equations of the type
\begin{align}\label{eq:diveq}
\begin{cases}
-\mathrm{div} \mathbb{A}(x,\nabla u) + \mathbb{V} |u|^{p-2}u &= \ -\mathrm{div}(|\mathbf{f}|^{p-2}\mathbf{f}) \quad \text{in} \ \ \Omega, \\
\hspace{2cm} u &=\ \ \mathsf{g}  \qquad \qquad \qquad \text{on} \ \ \partial \Omega,
\end{cases}
\end{align}
where the quasi-linear operator $\mathbb{A}$ is a Carath\'eodory function defined in $\Omega \times \mathbb{R}^n$ with valued in $\mathbb{R}^n$ and satisfied a version of $p$-monotone condition for some $p>1$. Here, we consider $\mathbb{A}$ under $(p,\sigma)$-monotone conditions for $p \in (1,n)$ and $\sigma \in [0,1]$, that means there exists a constant $\Upsilon>0$ satisfying
\begin{align}\label{eq:A1}
 \left| \mathbb{A}(x,\zeta) \right| & \le \Upsilon \left(\sigma^2 + |\zeta|^2 \right)^{\frac{p-1}{2}}, \\ \label{eq:A1b}
 |\partial_{\zeta} \mathbb{A}(x,\zeta)| & \le \Upsilon \left(\sigma^2 + |\zeta|^2 \right)^{\frac{p-2}{2}}, \\
\label{eq:A2}
\left( \mathbb{A}(x,\zeta_1)-\mathbb{A}(x,\zeta_2) \right) \cdot \left( \zeta_1 - \zeta_2 \right) & \ge \Upsilon^{-1} \left(\sigma^2 + |\zeta_1|^2 + |\zeta_2|^2 \right)^{\frac{p-2}{2}}|\zeta_1 - \zeta_2|^2,
\end{align}
for any $\zeta$, $\zeta_1$, $\zeta_2$ in $\mathbb{R}^n$  with a given constraint $\sigma^2 + |\zeta_1|^2 + |\zeta_2|^2 \neq 0$ and almost every $x$ in $\Omega$. Further on, we remark that, in the conditions considered above, $\sigma>0$ regards the non-degenerate problem and $\sigma=0$ for the degenerate case. A typical example of $\mathbb{A}$ satisfying the above conditions is $\mathbb{A}(x,\zeta) = (\sigma^2+|\zeta|^2)^{\frac{p-2}{2}}\zeta$ which gives rise to the degenerate/non-degenerate $p$-Laplacian equation~\eqref{eq:p-Lap}. Moreover, for the purpose of our current work, the  nonlinear operator $\mathbb{A}$ must obey the smallness condition of BMO type with respect to $x$, also known as \emph{the $(\delta,r)$-BMO condition}, this means there is a small number $\delta>0$ such that 
\begin{align}\label{cond:BMO}
[\mathbb{A}]^{r} = \sup_{y \in \mathbb{R}^n, \ 0<\varrho\le r} \fint_{B_{\varrho}(y)} \left( \sup_{\zeta \in \mathbb{R}^n \setminus \{0\}} \frac{|\mathbb{A}(x,\zeta) - \overline{\mathbb{A}}_{B_{\varrho}(y)}(\zeta)|}{(\sigma^2 +|\zeta|^2)^{\frac{p-1}{2}}} \right) dx \le \delta,
\end{align}
where $\overline{\mathbb{A}}_{B_{\varrho}(y)}(\zeta)$ stands for the average of $\mathbb{A}(\cdot,\zeta)$ over the ball $B_{\varrho}(y)$. It is remarkable that the class of operators satisfying such small BMO condition is larger than that of ones satisfying VMO condition. The $(\delta,r)$-BMO condition allows the nonlinearity $\mathbb{A}(x,\zeta)$ may be discontinuous in $x$ and this can be generally used as an appropriate substitute for the VMO condition originally defined by Sarason in \cite{Sarason}.

Regarding the potentials, we consider general non-negative Schr\"odinger term $\mathbb{V} \in L^1_{\mathrm{loc}}(\mathbb{R}^n; \mathbb{R}^+)$ satisfying $\mathbb{V} \in \mathbb{RH}^{\theta}$ for $\frac{n}{p} \le \theta < n$. It notices that $\mathbb{RH}^{\theta}$ denotes the reverse H{\"o}lder class containing locally integrable functions  that satisfy the \emph{reverse H\"older inequality}
\begin{align}\label{RH-ineq}
\displaystyle{\left(\fint_{B} \mathbb{V}^{\theta}(z) dz\right)^{\frac{1}{\theta}}} &\le C \displaystyle{\fint_{B} \mathbb{V}(z)dz}, 
\end{align}
for every ball $B \subset \mathbb{R}^n$. Furthermore, let us emphasize that $\mathbb{V}^{\frac{1}{p}}|\omega| \in L^p(\Omega)$ for all $\omega \in W^{1,p}(\Omega)$. Indeed, by denoting $\overline{\omega}_{\Omega}$ as the integral average of $\omega$ on $\Omega$ and noting that $1<  \frac{p\theta}{\theta-1} \le \frac{np}{n-p}$ for all $\theta \in [\frac{n}{p},n)$, we may apply H{\"o}lder and Sobolev's inequalities to get
\begin{align*}
\int_{\Omega} \mathbb{V}|\omega-\overline{\omega}_{\Omega}|^p dx & \le \left(\int_{\Omega} \mathbb{V}^{\theta}(x) dx\right)^{\frac{1}{\theta}}\left(\int_{\Omega} |\omega-\overline{\omega}_{\Omega}|^{\frac{p\theta}{\theta-1}} dx\right)^{\frac{\theta-1}{\theta}} \\
& \le C \left(\int_{\Omega} \mathbb{V}(x) dx\right)\left(\int_{\Omega} |\nabla \omega|^{p} dx\right). 
\end{align*}
In many issues related to physics, Schr\"odinger equations are usually governed by octic potential, decatic potential, polynomial potential, etc. A good example of a function that belongs to reverse H\"older class is $\mathbb{V}(x) = |x|^{-s} \in \mathbb{RH}^\theta$ for $s<\frac{n}{\theta}$ or when $\theta=\infty$, positive polynomials belong to $\mathbb{RH}^\theta$.

Due to the analysis of Calder\'on-Zygmund estimates proved by Lee and Ok in \cite[Theorem 2.3]{LO_schro}, it is reasonable to add one more assumption on potential $\mathbb{V}$ as follows
\begin{align}\label{eq:Morrey_V}
\|\mathbb{V}\|_{L^{\gamma; p\gamma}(\Omega)} := \sup_{0<\varrho< \mathrm{diam}(\Omega); \, \xi \in \Omega}{\varrho^{p-\frac{n}{\gamma}}}\|\mathbb{V}\|_{L^{\gamma}(B_{\varrho}(\xi)\cap\Omega)}  \le 1.
\end{align}

In order to obtain the global estimate in the non-smooth domain $\Omega$, our equations here defined in the Reifenberg flat domain. This is the minimal geometric requirement assumed on the boundary $\partial\Omega$ of domain to ensure the main results of the geometric analysis continue to hold true in $\Omega$. Roughly speaking, in our main results we will prove the regularity under an assumption that $\Omega$ is $(\delta,r)$-Reifenberg with a small value of $\delta$, that is: for any $x \in \partial \Omega$ and $\varrho \in (0,r]$, there is a coordinate system $\{\xi_1,\xi_2,...,\xi_n\}$ with origin at $x$ satisfying
\begin{align}\label{hyp:R}
B_{\varrho}(x) \cap \{\xi: \ \xi_n > \delta \varrho\} \subset B_{\varrho}(x) \cap \Omega \subset B_{\varrho}(x) \cap \{\xi: \ \xi_n > -\delta \varrho\}.
\end{align}
Here, we write the set $\{\xi: \ \xi_n > c\}$ instead of the set $\{\xi = (\xi_1, \xi_2, ..., \xi_n) \in \mathbb{R}^n: \ \xi_n > c\}$, for the sake of brevity. This type of domain is flat in the sense that its boundary is close to planes or hyperplanes at every small scale. The Reifenberg flat domain is more much general than a Lipschitz domain with sufficiently small Lipschitz constants. This class first appeared in the study of Plateau problem (see \cite{ER60}) by Reifenberg and later exploited by several authors (see \cite{BW2,SSB4, MT2010, LM1} and many references therein). 

As mentioned earlier, taking the advantage of fractional maximal operators $\mathbf{M}_\alpha$, our approach could be a new approach for deriving global gradient estimates and applicable to various types of nonlinear problems. Let us give hereafter the definition of fractional maximal operators and to make our strategy precise, we also include the \emph{boundedness property} of $\mathbf{M}_\alpha$ in Lemma \ref{lem:M_alpha} below (see~\cite{PNJDE,MPTNsub,PNnonuniform} for its detailed proof). For a given $\alpha \in [0,n]$, the fractional maximal operator $\mathbf{M}_\alpha$ is defined by
\begin{align}\label{eq:Malpha}
\mathbf{M}_\alpha h(y) = \sup_{\varrho>0}{\varrho^\alpha \fint_{B_{\varrho}(y)}{|h(x)|dx}}, \quad y \in \mathbb{R}^n, \quad h \in L^1_{\mathrm{loc}}(\mathbb{R}^n;\mathbb{R}^+).
\end{align}
When $\alpha=0$, the operator $\mathbf{M}_0$ is well-know as the Hardy-Littlewood function $\mathbf{M}$ below
\begin{align*}
\mathbf{M}h(y) = \sup_{\varrho>0}{\fint_{B_{\varrho}(y)}|h(x)|dx}, \quad y \in \mathbb{R}^n.
\end{align*}
\begin{lemma}\label{lem:M_alpha}
Let $s \ge 1$ and $0 \le \alpha s < n$, there holds
\begin{align*}
\mathcal{L}^n\left(\left\{x \in \mathbb{R}^n: \ \mathbf{M}_{\alpha}h(x)> \lambda\right\}\right) \le C \left(\frac{1}{\lambda^{s}}\int_{\mathbb{R}^n}|h(y)|^s dy\right)^{\frac{n}{n-\alpha s}},
\end{align*}
for any $\lambda >0$ and $h \in L^s(\mathbb{R}^n)$.
\end{lemma}
Here we write $\mathcal{L}^n(K)$ for the Lebesgue measure in $\mathbb{R}^n$ of a measurable set $K \subset \mathbb{R}^n$. 

In the present article, we investigate the regularity estimates in the setting of Lorentz space $L^{q,s}(\Omega)$. We now recall the definition of Lorentz spaces - the real interpolation spaces between Lebesgue spaces.  For $0<q<\infty$ and $1<s \le \infty$, the Lorentz spaces $L^{q,s}(\Omega)$ is the set containing all of functions $h \in L^1_{\mathrm{loc}}(\Omega)$ such that 
\begin{align}\label{Lor-norm}
\|h\|_{L^{q,s}(\Omega)} := \left[ q \int_0^\infty \lambda^{s-1} \mathcal{L}^n(\{\xi \in \Omega: |h(\xi)|>\lambda\} )^{\frac{s}{q}} d\lambda \right]^{\frac{1}{s}} < \infty, 
\end{align}
if $0< q < \infty$ and $0 < s < \infty$ and
\begin{align}\label{Lor-norm-inf}
\|h\|_{L^{q,\infty}(\Omega)} :=  \sup_{\lambda>0}{\lambda \mathcal{L}^n(\{\xi \in \Omega:|h(\xi)|>\lambda\})^{\frac{1}{q}}} < \infty,
\end{align}
if $s = \infty$. For this case, the space $L^{q,\infty}$ is the usual weak-$L^q$ or Marcinkiewicz space.

For simplicity of notation in our statements throughout the paper, we will use a new mapping $\Psi_{\sigma}: \ W^{1,p}(\Omega) \to \mathbb{R}^+$ defined by
\begin{align}\label{def:Psi}
\Psi_{\sigma}(\omega) & = \sigma^p + |\nabla \omega|^p + \mathbb{V}|\omega|^p, \quad \omega \in W^{1,p}(\Omega).
\end{align}

Related to the regularity results for degenerate case when $\sigma=0$, we simply write $\Psi$ instead of $\Psi_0$.  Here, almost the statements of our work require the same assumptions that $\Omega$ is $(\delta,r_0)$-Reifenberg in~\eqref{hyp:R} and $\mathbb{A}$ satisfies the $(\delta,r_0)$-BMO condition~\eqref{cond:BMO} for $\delta>0$ small enough. For reasons of brevity and deliberately not repeated, we only use the notation $(\mathcal{H}_{\delta})$ for both assumptions meanwhile they satisfy for some $r_0>0$. 

Let us now state our main results in this paper.

\begin{theorem}\label{theo:main-A}
For every $0 \le \alpha < \frac{n}{\theta}$ and $a > \frac{1}{\theta} - \frac{\alpha}{n} $, one can choose constants $\delta = \delta(\theta,\Upsilon,n,p)>0$, $b = b(a,\theta,\alpha,n,p) >0$ and $\varepsilon_0 = \varepsilon_0(a,b,n) \in (0,1)$ such that under hypothesis $(\mathcal{H}_{\delta})$ the following estimate
\begin{align}\label{est:A}
d_{\mathbf{M}_{\alpha}(\Psi_{\sigma}(u))}(\varepsilon^{-a}\lambda) \le C\varepsilon d_{\mathbf{M}_{\alpha}(\Psi_{\sigma}(u))}(\lambda) + d_{\mathbf{M}_{\alpha}(|\mathbf{f}|^p + \Psi_{\sigma}(\mathsf{g}))}(\varepsilon^{b}\lambda),
\end{align}
is valid for $\lambda > 0$ and $0< \varepsilon <\varepsilon_0$. Here, the distribution function $d_f$ is defined by the Lebesgue measure of a certain set
\begin{align*}
d_f(\lambda) = \mathcal{L}^n \left(\left\{x \in \Omega: \ |f(x)|> \lambda\right\}\right), \quad \lambda \ge 0,
\end{align*} 
that will be clarified in Section \ref{sec:level} later, and $C = C(\alpha,a,\theta,\Upsilon,n,p,\mathrm{diam}(\Omega)/r_0) > 0.$  
\end{theorem}
\begin{theorem}\label{theo:main-B}
For any $0 \le \alpha < \frac{n}{\theta}$, one can find $\delta = \delta(\theta,\Upsilon,n,p)>0$ such that under hypothesis $(\mathcal{H}_{\delta})$ then
\begin{align}\label{est:B}
\| \mathbf{M}_{\alpha}(\Psi_{\sigma}(u))\|_{L^{q,s}(\Omega)} \le C \|\mathbf{M}_{\alpha}(|\mathbf{f}|^p + \Psi_{\sigma}(\mathsf{g})) \|_{L^{q,s}(\Omega)},
\end{align}
holds for any $0<q<\frac{n\theta}{n-\alpha\theta}$ and $0<s \le \infty$, where $C = C(\alpha,q,s,\Upsilon,n,p,\mathrm{diam}(\Omega)/r_0)>0$.
\end{theorem}
Thanks to the boundedness property of maximal operator on the Lorentz spaces, one easily obtain the following straightforward corollary.
\begin{corollary}\label{coro}
There exists $\delta = \delta(\theta,\Upsilon,n,p)>0$ small enough such that under hypothesis $(\mathcal{H}_{\delta})$ then
\begin{align}\label{est:BC}
\| \sigma + |\nabla u|\|_{L^{q,s}(\Omega)} \le C \| \sigma + |\mathbf{f}| + |\nabla \mathsf{g}| + \mathbb{V}^{\frac{1}{p}} |\mathsf{g}| \|_{L^{q,s}(\Omega)},
\end{align}
for any $0<q<p\theta$ and $0<s \le \infty$.
\end{corollary}

Let us review some important technical tools in Calder\'on-Zygmund and regularity estimates for nonlinear problems that developed over the last years.  By deeply using the interaction between Harmonic analysis and nonlinear partial differential equations, a number of intensive studies have been investigated by Iwaniec \cite{Iwaniec}, Caffarelli and Peral \cite{CP1998}, Dong and Kim \cite{DK2011}, Krylov \cite{Krylov} and so on. In the past few years, there have been strong efforts to obtain the global gradient estimates or Calder\'on-Zygmund type estimates over the nonsmooth domains and the geometrical approach by Byun and Wang was very impressive via number of results \cite{BW2,SSB4}. Another extremely important technique with no use of Harmonic Analysis allows to prove higher integrability or Calder\'on-Zygmund results has been introduced by Acerbi and Mingione in \cite{AM2007}. Nowadays, such amazing technique becomes a standard standard tools in yielding locally regularity results for a large class of nonlinear problems. Besides, a number of authors have continued, combined and developed gradient/Calder\'on-Zygmund estimates to various nonlinear elliptic and parabolic problems over non-smooth domains. For instance, \cite{55QH4, Mi3,MPT2018,PNCRM,PNJDE,MPTNsub,Duzamin2,KM2012, KM2014} and needless to say, the list of references is incomplete.

Throughout this paper, we employ the technique in \cite{AM2007, Mi3} in a different point of view, first presented in \cite{PN_dist}. An interesting feature of our approach is the appearance of \emph{fractional maximal distribution functions} (FMD) and the constructions of level-set inequalities in proofs. Apart from its own interest, our approach is useful for understanding the essence behind the proofs of Calder\'on-Zygmund-type estimates and further on, it enables us to make the use of this technique to obtain other regularity results, especially in terms of fractional maximal functions $\mathbf{M}_\alpha$ (see \cite{PN14}).

The remainder of this paper is organized as follows. In Section \ref{sec:inter_bound}, a few preliminary results, in which local comparison estimates will be stated and proved. To handle the difficulties of proof, we revisit the approach of level-set inequalities on FMDs and the next section is concerned with establishing those types of level-set inequalities. The approach of working with level-set inequalities on FMDs was first proposed in our previous work \cite{PN_dist} and becomes an effective tool in proving the global regularity results via fractional maximal operators. And finally, Section \ref{sec:proofs} is devoted to proving our main results of this paper.

\section{Comparison estimates}\label{sec:inter_bound}

We first recall the classical global estimate for gradient of solutions to~\eqref{eq:diveq} in Lebesgue space $L^p(\Omega)$. Following the idea earlier introduced by Mingione in \cite{Mi3}, we next establish the local estimates related to the weak solutions to the corresponding homogeneous problem. 

In what follows, we always consider $u$ as a weak solution to~\eqref{eq:diveq} under $(p,\sigma)$-monotone conditions of $\mathbb{A}$ for $1< p < n$ and $0 \le \sigma \le 1$, with given data $\mathbf{f} \in L^p(\Omega; \mathbb{R}^n)$ and boundary condition $\mathsf{g} \in W^{1,p}(\Omega)$. We recall that $u \in W^{1,p}(\Omega)$ is a weak (distribution) solution to~\eqref{eq:diveq} if it satisfies the following variational formula
\begin{align}\label{var-form}
\int_\Omega{\mathbb{A}(x,\nabla u) \cdot \nabla\varphi  dx} + \int_{\Omega} \mathbb{V} |u|^{p-2}u \varphi dx = \int_\Omega{ |\mathbf{f}|^{p-2} \mathbf{f} \cdot \nabla\varphi  dx},
\end{align}
for any $\varphi \in W_0^{1,p}(\Omega)$.

\subsection{A global Lebesgue estimate}
\begin{lemma}\label{lem:Global}
One can find a constant $C(\Upsilon,\sigma,n,p)>0$ such that
\begin{equation}\label{est:Global}
\int_{\Omega} \Psi_{\sigma}(u) dx \le C(\Upsilon,\sigma,n,p) \int_{\Omega}\left(|\mathbf{f}|^p + \Psi_{\sigma}(\mathsf{g})\right)dx.
\end{equation}
\end{lemma}
\begin{proof}
For simplicity, we now use the following notation
\begin{align*}
J_{\sigma}(\psi) = (\sigma^2 + |\psi|^2)^{\frac{p-2}{2}}|\psi|^2, \quad \psi \in \mathbb{R}^n, \ \sigma^2 + |\psi|^2 \neq 0.
\end{align*}
Let us choose $\varphi = u-\mathsf{g}$ in~\eqref{var-form} to find out
\begin{align*}
\int_{\Omega} \mathbb{A}(x, \nabla u) \cdot \nabla u dx + \int_{\Omega}\mathbb{V}|u|^pdx & = \int_{\Omega} \mathbb{A}(x, \nabla u) \cdot \nabla \mathsf{g} dx \\
& \ + \int_{\Omega}\mathbb{V}|u|^{p-2}u \mathsf{g} dx + \int_{\Omega}|\mathbf{f}|^{p-2}\mathbf{f} \cdot \nabla(u-\mathsf{g})dx.
\end{align*}
Thanks to \eqref{eq:A1} and \eqref{eq:A2}, this equality implies to
\begin{align*}
& \int_{\Omega} \left(J_{\sigma}(\nabla u) + \mathbb{V}|u|^p \right) dx  \leq C(\Upsilon)\left(\int_{\Omega}(\sigma^2 + |\nabla u|^2)^{\frac{p-1}{2}}|\nabla \textsf{g}|dx  + \int_{\Omega} \sigma^{p-1} |\nabla u|dx \right. \\
 & \hspace{3cm} \left.  + \int_{\Omega}\mathbb{V}^{1-\frac{1}{p}}|u|^{p-1} \mathbb{V}^{\frac{1}{p}}|\mathsf{g}| dx + \int_{\Omega}(|\nabla u|+ |\nabla \mathsf{g}|)|\mathbf{f}|^{p-1} dx\right).
\end{align*}
For every $\varepsilon_1 > 0$, one may apply H{\"o}lder and Young's inequalities to arrive 
\begin{align} \label{est:01}
\int_{\Omega} \left( J_{\sigma}(\nabla u) + \mathbb{V}|u|^p \right) dx   \leq \varepsilon_1 \int_{\Omega} \Psi_{\sigma}(u)dx +  C(\varepsilon_1)\int_{\Omega} (|\mathbf{f}|^p + \Psi_{\sigma}(\mathsf{g}))dx.
\end{align}
Remark that $|\nabla u|^p \le J_{\sigma}(\nabla u)$ when $p \ge 2$, it allows us to obtain~\eqref{est:Global} from~\eqref{est:01} by taking $\varepsilon_1 = \frac{1}{2}$ in this case. Otherwise when $p \in (1,2)$, we first present $|\nabla u|^p$ as follows
\begin{align*}
|\nabla u|^p & =   \left(J_{\sigma}(\nabla u)\right)^{\frac{p}{2}} (\sigma^2 + |\nabla u|^2)^{\frac{p}{2}\left(1-\frac{p}{2}\right)}  \le \left(2J_{\sigma}(\nabla u)\right)^{\frac{p}{2}} (\sigma^p + |\nabla u|^p)^{1-\frac{p}{2}},
\end{align*}
and we then apply Young inequality to get that
\begin{align}\label{est:0100}
\int_{\Omega} |\nabla u|^p dx &\le \varepsilon_2 \int_{\Omega} (\sigma^p + |\nabla u|^p) dx + 2\varepsilon_2^{1-\frac{2}{p}} \int_{\Omega} J_{\sigma}(\nabla u) dx,
\end{align}
for every $\varepsilon_2 > 0$. Combining between~\eqref{est:01} and~\eqref{est:0100}, one has
\begin{align*}
\int_{\Omega} \Psi_{\sigma}(u) dx &\le  \int_{\Omega} (\sigma^p + \varepsilon_2|\nabla u|^p) dx  + C\varepsilon_2^{1-\frac{2}{p}} \int_{\Omega} \left(J_{\sigma}(\nabla u) + \mathbb{V}|u|^p \right)dx \\ 
 &\le \left(\varepsilon_2 +  C\varepsilon_2^{1-\frac{2}{p}}\varepsilon_1\right)\int_{\Omega}  \Psi_{\sigma}(u) dx + C(\varepsilon_1, \varepsilon_2)\int_{\Omega} (|\mathbf{f}|^p +  \Psi_{\sigma}(\mathsf{g}) )dx.
\end{align*}
We may conclude~\eqref{est:Global} by choosing $\varepsilon_1 = \frac{1}{C}(\frac{1}{2} - \varepsilon_2)\varepsilon_2^{\frac{2}{p}-1}$ and $\varepsilon_2 = \frac{2-p}{2}$ in the last computation.
\end{proof}
\subsection{Local comparison estimates}
\label{sec:local}

In this subsection, we employ the homogeneous Dirichlet problems to construct and prove the comparison results between distribution solutions to the original problem \eqref{eq:diveq} and the unique solutions to suitable reference problems. Proofs of these comparison arguments to assert the gradient estimates of solutions to original problems were originally based on the ideas of Mingione \textit{et al.}, going back at least to \cite{Mi4, KM2014,Duzamin2}. And Lemma \ref{lem:Comp} provides an important result for the main proof of gradient estimates. To stress further this section, we derive both comparison estimates in the interior of our domain and up to the boundary. 

The auxiliary Lemmas \ref{lem:RH} and \ref{lem:RH-boundary} allow us to obtain the gradient estimates of solutions to corresponding homogeneous problems in the interior of domain and on the boundary, returning to the work by Lee and Ok in~\cite{LO_schro}.  We further remark here that although in~\cite{LO_schro}, authors only proved the reverse H\"older type inequality concerning the case $\sigma=0$, it allows us to treat the validity for the case $\sigma>0$ by applying the similar arguments. 

\begin{lemma} \label{lem:RH}
Let $x_0 \in \Omega$ and $R \in (0,r_0/2]$. Suppose that $v$ solves the following homogeneous problem
\begin{equation}\label{eq:I1-a}
-\mathrm{div}  \mathbb{A}(x,\nabla v)  + \mathbb{V} |v|^{p-2}v  = \ 0,  \quad \mbox{ in } B_{2R}(x_0).
\end{equation}
Assume moreover that $\mathbb{V} \in \mathbb{RH}^{\theta}$ for some $\theta \in [\frac{n}{p}, n)$ satisfying $\|\mathbb{V}\|_{L^{\theta;p\theta}}(\Omega) \le 1$. If $[\mathbb{A}]^{r_0} \le \delta$ for $\delta$ small enough then there holds
\begin{equation}\label{est:RH}
\left(\fint_{B_{R}(x_0)} (\Psi_{\sigma}(v))^{\theta} dx\right)^{\frac{1}{\theta}}\leq C\fint_{B_{2R}(x_0)} \Psi_{\sigma}(v) dx. 
\end{equation}
\end{lemma}

\begin{lemma}\label{lem:RH-boundary}
Let $x_0 \in \partial\Omega$, $R \in (0,r_0/4]$ and $\Omega_{4R}=B_{4R}(x_0) \cap \Omega$. Suppose that $\tilde{v}$ solves the following problem
\begin{equation}\label{eq:RH-v-b}
\begin{cases}-\mathrm{div} \mathbb{A}(x,\nabla \tilde{v}) + \mathbb{V} |\tilde{v}|^{p-2} \tilde{v} & = \ 0, \hspace{1.5cm} \mbox{ in } \Omega_{4R}(x_0), \\
\hspace{2cm} \tilde{v} & = u - \mathsf{g}, \quad \ \quad \mbox{ on } \partial\Omega_{4R}(x_0).
\end{cases}
\end{equation}
Assume moreover that $\mathbb{V} \in \mathbb{RH}^{\theta}$ for some $\theta \in [\frac{n}{p}, n)$ satisfying $\|\mathbb{V}\|_{L^{\theta;p\theta}}(\Omega) \le 1$. Under hypothesis $(\mathcal{H}_{\delta})$ for $\delta$ small enough then there holds    
\begin{equation}\label{est:RH-b}
\left(\fint_{\Omega_{R}(x_0)}(\Psi_{\sigma}(\tilde{v}))^{\theta} dx\right)^{\frac{1}{\theta}}\leq C \fint_{\Omega_{4R}(x_0)} \Psi_{\sigma}(\tilde{v}) dx.
\end{equation}
\end{lemma}

\begin{lemma} \label{lem:Comp}
Let $x_0 \in \overline{\Omega}$, $R \in (0,r_0/2]$ and $\Omega_{2R}=B_{2R}(x_0) \cap \Omega$. Assume that $v$ solves the following homogeneous problem
\begin{equation}\label{eq:I1}
\begin{cases} -\mathrm{div}  \mathbb{A}(x,\nabla v)  + \mathbb{V} |v|^{p-2}v & = \ 0, \quad \ \quad \mbox{ in } \Omega_{2R},\\ 
\hspace{1.2cm} v & = \ u - \mathsf{g}, \ \mbox{ on } \partial \Omega_{2R}.\end{cases}
\end{equation}
One can find positive constants $k=k(p)>0$ and $C = C(\Upsilon,\sigma,n,p)>0$ such that for any $\varepsilon \in (0,1)$ there holds
\begin{align} \label{est:Comp}
\fint_{\Omega_{2R}} \Psi(u-v) dx  & \le \varepsilon  \fint_{\Omega_{2R}} \Psi_{\sigma}(u) dx  +  C\varepsilon^{-k} \fint_{\Omega_{2R}} \left( |\mathbf f|^p + \Psi(\mathsf{g})\right) dx. 
\end{align}
\end{lemma}
\begin{proof}
Let us test equation~\eqref{eq:I1} by $u-v-\mathsf{g}$ and choose $\varphi$ by this function in \eqref{var-form}. Subtracting these observed formulas, one obtains that
\begin{align}\label{est:200}
& \int_{\Omega_{2R}} \left(\mathbb{A}(x,\nabla u)-\mathbb{A}(x, \nabla v)\right) \cdot \nabla (u-v) dx + \int_{\Omega_{2R}}\mathbb{V}(|u|^{p-2}u-|v|^{p-2}v)(u-v) dx \notag \\ 
& \hspace{1cm} =  \int_{\Omega_{2R}} \left(\mathbb{A}(x,\nabla u)-\mathbb{A}(x, \nabla v) \right) \cdot \nabla \mathsf{g}  dx + \int_{\Omega_{2R}}\mathbb{V}(|u|^{p-2}u-|v|^{p-2}v)\mathsf{g} dx  \notag \\ 
& \hspace{3cm} + \int_{\Omega_{2R}}|\mathbf f|^{p-2}\mathbf{f} \cdot (\nabla u - \nabla v)dx - \int_{\Omega_{2R}}|\mathbf{f}|^{p-2}\mathbf{f} \cdot \nabla \mathsf{g}dx .
\end{align}
Let us introduce a new function $\Phi_{\sigma}: \ (W^{1,p}(\Omega))^2 \to \mathbb{R}^+$ determined by
\begin{align}\label{def:Phi}
\Phi_{\sigma}(\varphi, \psi) & := (\sigma^2 + |\nabla \varphi|^2 + |\nabla \psi|^2)^{\frac{p-2}{2}}|\nabla (\varphi - \psi)|^2 \notag \\ & \hspace{3cm} + \mathbb{V}  \left(|\varphi|^2 + |\psi|^2\right)^{\frac{p-2}{2}} |\varphi -\psi|^2, 
\end{align}
for all $\varphi$ and $\psi \in W^{1,p}(\Omega)$. We recall the following classical relation 
$$\big||\varphi|^{p-2}\varphi-|\psi|^{p-2}\psi\big| \sim \left(|\varphi|^2 + |\psi|^2\right)^{\frac{p-2}{2}}|\varphi-\psi|,$$ 
whose proof can be seen as a special case of~\cite[Lemma 2.1]{H92} for instance. Combining the above inequality with the definition of $\Phi_{\sigma}$ in~\eqref{def:Phi} and assumptions~\eqref{eq:A1}, \eqref{eq:A2}, it implies from~\eqref{est:200} that 
\begin{align} \label{est:201}
 \fint_{\Omega_{2R}} \Phi_{\sigma}(u,v) dx & \leq C \fint_{\Omega_{2R}}\left[(\sigma^2 + |\nabla u|^2)^{\frac{p-1}{2}} + (\sigma^2 + |\nabla v|^2)^{\frac{p-1}{2}}\right] |\nabla \mathsf{g}| dx \notag \\ 
& \hspace{1cm} +  \fint_{\Omega_{2R}} \mathbb{V}  \left(|u|^2 + |v|^2\right)^{\frac{p-2}{2}} |u -v| |\mathsf{g}| dx \notag \\ 
& \hspace{2cm} + \fint_{\Omega_{2R}} |\nabla u -\nabla v||\mathbf{f} |^{p-1}dx + \fint_{\Omega_{2R}}|\mathbf{f}|^{p-1}|\nabla \mathsf{g}|dx .
\end{align}
Let us now estimate the middle-term on the right-hand side of~\eqref{est:201}. For $1 < p \le 2$, it is easy to see that
\begin{align*}
\fint_{\Omega_{2R}} \mathbb{V} \left(|u|^2 + |v|^2\right)^{\frac{p-2}{2}} |u -v| |\mathsf{g}| dx \le C(p) \fint_{\Omega_{2R}} \mathbb{V} |u-v|^{p-1} |\mathsf{g}| dx.
\end{align*}
Otherwise, when $p>2$, then $ \left(|u|^2 + |v|^2\right)^{\frac{p-2}{2}} \le C(p) \left(|u-v|^{p-2} + |u|^{p-2}\right)$ which yields
\begin{align*}
\fint_{\Omega_{2R}} \mathbb{V} \left(|u|^2 + |v|^2\right)^{\frac{p-2}{2}} |u -v| |\mathsf{g}| dx & \le C(p) \left(\fint_{\Omega_{2R}} \mathbb{V} |u-v|^{p-1} |\mathsf{g}| dx \right. \\
& \hspace{3cm} \left. + \fint_{\Omega_{2R}} \mathbb{V} |u|^{p-2} |u-v| |\mathsf{g}| dx. \right)
\end{align*}
Taking two above inequalities into account, we are able to estimate the terms on the right-hand side of~\eqref{est:201} by applying the following fundamental inequality
$$(\sigma^2+|\varphi_2|^2)^{\frac{p-1}{2}} \leq C(p) \left(\sigma^{p-1} + |\varphi_1|^{p-1} + |\varphi_1 -\varphi_2|^{p-1}\right),$$
and the consequences of Young's inequality for every $\varepsilon, \varepsilon_1, \varepsilon_2>0$ as below
\begin{align*}
|\varphi_1|^{p-1} |\varphi_2| \le \varepsilon |\varphi_1|^p +  \varepsilon^{1-p} |\varphi_2|^p, \quad |\varphi_1| |\varphi_2|^{p-1} \le \varepsilon |\varphi_1|^p + C \varepsilon^{\frac{1}{1-p}} |\varphi_2|^p,
\end{align*}
for $p>1$ and
\begin{align*}
 |\varphi_1|^{p-2} |\varphi_2| |\varphi_3| \le \varepsilon_1 |\varphi_1|^p + \varepsilon_2 |\varphi_2|^p + C \varepsilon_1^{2-p} \varepsilon_2^{-1} |\varphi_3|^p, \quad \mbox{for } p > 2.
\end{align*}
Applying these inequalities, from \eqref{est:201} it allows us to arrive
\begin{align} 
 \fint_{\Omega_{2R}} \Phi_{\sigma}(u, v)dx & \leq C \left[\varepsilon_1 \fint_{\Omega_{2R}}\left(\sigma^p + |\nabla u|^{p} +  \mathbb{V} |u|^p \right)dx + \varepsilon_1^{1-p}\fint_{\Omega_{2R}}|\nabla \mathsf{g}|^p dx \right. \notag \\ 
& \hspace{2cm} \left.  + \varepsilon_2\fint_{\Omega_{2R}} \Psi(u-v) dx +  \varepsilon_2^{\frac{1}{1-p}}\fint_{\Omega_{2R}}|\mathbf{f}|^{p}dx \right. \notag \\
& \hspace{3cm} \left. + \left(\varepsilon_2^{1-p} +  \varepsilon_1^{2-p} \varepsilon_2^{-1} \right)\fint_{\Omega_{2R}} \Psi(\mathsf{g}) dx \right], \notag
\end{align}
which can be rewritten as
\begin{align} \label{est:02}
& \fint_{\Omega_{2R}} \Phi_{\sigma}(u, v)dx  \leq C \left[\varepsilon_1 \fint_{\Omega_{2R}} \Psi_{\sigma}(u) dx + \varepsilon_2\fint_{\Omega_{2R}} \Psi(u-v) dx  \right. \notag \\ 
& \hspace{2cm} \left.  + \left(\varepsilon_2^{\frac{1}{1-p}} + \varepsilon_1^{1-p} + \varepsilon_2^{1-p} +  \varepsilon_1^{2-p} \varepsilon_2^{-1}\right)\fint_{\Omega_{2R}} |\mathbf{f}|^{p} + \Psi(\mathsf{g}) dx \right],
\end{align}
for every $\varepsilon_1, \, \varepsilon_2 \in (0,1)$, where $\Psi_{\sigma}$ and $\Psi$ are defined as in~\eqref{def:Psi}. Using the same technique as the proof of~\eqref{est:0100}, one has no difficulty to show 
\begin{align}\label{Phi-ineq}
\fint_{\Omega_{2R}} \Psi(u-v) dx \le \varepsilon_3 \fint_{\Omega_{2R}} \Psi_{\sigma}(u) dx + C \varepsilon_3^{-\vartheta} \fint_{\Omega_{2R}} \Phi_{\sigma}(u, v)dx,
\end{align}
for all $\varepsilon_3>0$ with $\vartheta = \max\left\{0,\frac{2}{p}-1\right\}$. Let us now substitute~\eqref{est:02} into~\eqref{Phi-ineq} to observe
\begin{align}\label{est:202}
& \fint_{\Omega_{2R}} \Psi(u-v) dx \le \left(\varepsilon_3 + C \varepsilon_1 \varepsilon_3^{-\vartheta}\right) \fint_{\Omega_{2R}} \Psi_{\sigma}(u) dx + C \varepsilon_2 \varepsilon_3^{-\vartheta}  \fint_{\Omega_{2R}} \Psi(u-v) dx \notag \\
& \hspace{2cm}  + C\left(\varepsilon_2^{\frac{1}{1-p}} + \varepsilon_1^{1-p} + \varepsilon_2^{1-p} +  \varepsilon_1^{2-p} \varepsilon_2^{-1}\right) \varepsilon_3^{-\vartheta}\fint_{\Omega_{2R}} |\mathbf{f}|^{p} + \Psi(\mathsf{g}) dx.
\end{align}
The remaining point concerns to choose suitable values of $\varepsilon_1$, $\varepsilon_2$ and $\varepsilon_3$ in~\eqref{est:202} depending on an arbitrary number $\varepsilon \in (0,1)$. More precisely, one can choose
\begin{align*}
\varepsilon_3 = \frac{\varepsilon}{4}, \quad C \varepsilon_2 \varepsilon_3^{-\vartheta} \le \frac{1}{2}, \quad \varepsilon_3 + C \varepsilon_1 \varepsilon_3^{-\vartheta} \le \frac{\varepsilon}{2},
\end{align*}
which allows us to conclude~\eqref{est:Comp} with $k$ determined by
$$k = \vartheta + \max\left\{\left(1+\vartheta\right)(p-1); \ \frac{\vartheta}{p-1}\right\}>0.$$
It finishes the proof.
\end{proof}

\section{Level-set inequalities on distribution function}
\label{sec:level}
In this section, we establish several level-set inequalities related to the distribution function of measurable functions which is considered in our previous works such as~\cite{PN_dist}. For every measurable function $f$ on $\Omega$ and $K \subset \mathbb{R}^n$, the distribution function $d_f(K,\cdot)$ of $f$ is defined in $\mathbb{R}^+$ as follows
\begin{align}\label{def-Df}
d_f(K,\lambda) = \mathcal{L}^n \left(\left\{x \in K \cap \Omega: \ |f(x)|> \lambda\right\}\right), \quad \lambda \ge 0.
\end{align}
If $\Omega \subset K$, we write $d_f(\lambda)$ instead of $d_f(\Omega,\lambda)$ for simplicity. 

\begin{lemma}\label{lem:A1}
Let $\alpha \in [0,n)$ and $a >0$. One can choose $b=b(\alpha,a,n)>0$ and $\varepsilon_0 = \varepsilon_0(\alpha,a,b,n,\mathrm{diam}(\Omega)/R)>0$ such that if provided $x_1 \in \Omega$ satisfying $\mathbf{M_\alpha}(|\mathbf{f}|^p + \Psi_{\sigma}(\mathsf{g}))(x_1) \leq \varepsilon^b \lambda$ for some $\varepsilon \in (0,\varepsilon_0)$, $\lambda>0$ then for every $R>0$, there holds
\begin{align}\label{est:4.2}
d_{\mathbf{M}_\alpha(\Psi_{\sigma}(u))}(\Omega, \varepsilon^{-a} \lambda) \leq  \varepsilon \mathcal{L}^n(B_R(0)).
\end{align}
\end{lemma}
\begin{proof}
With distribution function $d_f$ defined as in~\eqref{def-Df}, combining between the boundedness of fractional maximal function $\mathbf{M}_\alpha$ and the global estimate~\eqref{est:Global} given as in Lemma~\ref{lem:Global}, one gets that 
\begin{align}\label{est:4.2.1}
d_{\mathbf{M}_\alpha(\Psi_{\sigma}(u))}(\varepsilon^{-a}\lambda) \leq C\left(\frac{1}{\varepsilon^{-a}\lambda}\int_{\Omega} \Psi_{\sigma}(u) dx \right)^{\frac{n}{n-\alpha}} \leq C\left(\frac{1}{\varepsilon^{-a}\lambda}\int_{\Omega} (|\mathbf{f}|^p + \Psi_{\sigma}(\mathsf{g})) dx \right)^{\frac{n}{n-\alpha}}.
\end{align}
Let us remind the fact $\mathbf{M}_\alpha(|\mathbf{f}|^p + \Psi_{\sigma}(\mathsf{g}))(x_1) \le \varepsilon^{b}\lambda$ and $\Omega \subset Q := B_{D_0}(x_1)$ with $D_0=\mathrm{diam}(\Omega)$, it implies from~\eqref{est:4.2.1} that 
\begin{align}\label{est:4.2.20}
d_{\mathbf{M}_\alpha(\Psi_{\sigma}(u))}(\varepsilon^{-a}\lambda) &\le C\left(\frac{\mathcal{L}^n(Q)}{\varepsilon^{-a}\lambda}D_0^{-\alpha}\mathbf{M}_\alpha(|\mathbf{f}|^p + \Psi_{\sigma}(\mathsf{g}))(x_1) \right)^{\frac{n}{n-\alpha}} \notag \\
& \le C \left(\frac{D_0}{R}\right)^n \varepsilon^{\frac{(a+b)n}{n-\alpha}} \mathcal{L}^n(B_R(0)).
\end{align} 
For every $\alpha \in [0,n)$ and $a >0$, we may choose $b$ in~\eqref{est:4.2.20} and $\varepsilon_0>0$ such that
\begin{align*}
b > \max\left\{0; \ 1- a - \frac{\alpha}{n}\right\} \ \mbox{ and } \ C \left(\frac{D_0}{R}\right)^n \varepsilon_0^{\frac{(a+b)n}{n-\alpha}-1} < 1.
\end{align*}
That leads to~\eqref{est:4.2} for all $\varepsilon \in (0,\varepsilon_0)$.
\end{proof}

\begin{lemma}\label{lem:A2}
Let $\alpha \in [0,n)$ and $x_2 \in \Omega_{R}(x_0)$ satisfying $\mathbf{M}_\alpha(\Psi_{\sigma}(u))(x_2) \le \lambda$ for some $\lambda>0$. Then for every $a >0$ the following inequality 
\begin{align} \label{est:4.3}
d_{\mathbf{M}_{\alpha}(\Psi_{\sigma}(u))}(\Omega_{R}(x_0), \varepsilon^{-a}\lambda) \le d_{\mathbf{M}_{\alpha}(\chi_{B_{2R}(x_0)}\Psi_{\sigma}(u))}(\Omega_{R}(x_0), \varepsilon^{-a}\lambda),
\end{align}
holds for all $0< \varepsilon \le 3^{-\frac{n+1}{a}}$.
\end{lemma}
\begin{proof}
Let us first introduce the cut-off version corresponding to the fractional maximal function $\mathbf{M}_{\alpha}$ of $f$ at the cut-off order $R>0$ as below
$$ \mathbf{M}_{\alpha}^{R} f(z) = \sup_{0<r<R} r^{\alpha}\fint_{B_r(z)}f(x)dx, \ \mbox{ and } \ \mathbf{T}_{\alpha}^{R} f(x) = \sup_{r \ge R} r^{\alpha}\fint_{B_r(z)}f(x)dx.$$
Using this notation, for every $z$ in $B_{R}(x_0)$ we can separate
\begin{align*}
\mathbf{M}_{\alpha}(\Psi_{\sigma}(u))(z) \leq \max \left\{ \mathbf{M}_{\alpha}^{R}(\Psi_{\sigma}(u))(z); \ \mathbf{T}_{\alpha}^{R}(\Psi_{\sigma}(u))(z)\right\},
\end{align*}
which gives us the following estimate
\begin{align}\label{est:4.3.2}
d_{\mathbf{M}_{\alpha}(\Psi_{\sigma}(u))}(\Omega_{R}(x_0), \varepsilon^{-a}\lambda) \le d_{\mathbf{M}_{\alpha}^{R}(\Psi_{\sigma}(u))}(\Omega_{R}(x_0), \varepsilon^{-a}\lambda) + d_{\mathbf{T}_{\alpha}^{R}(\Psi_{\sigma}(u))}(\Omega_{R}(x_0), \varepsilon^{-a}\lambda).
\end{align}
For all $r \ge R$, one can check $B_r(z) \subset B_{3r}(x_2)$ to point out
\begin{align*}
\mathbf{T}_{\alpha}^{R}(\Psi_{\sigma}(u))(z) & \le \sup_{r \ge R} r^{\alpha}\frac{\mathcal{L}^n(B_{3r}(x_2))}{\mathcal{L}^n(B_r(z))} \fint_{B_{3r}(x_2)}\Psi_{\sigma}(u)dx  \le 3^n \mathbf{M}_{\alpha}(\Psi_{\sigma}(u))(x_2),
\end{align*}
which yields that 
$$d_{\mathbf{T}_{\alpha}^{R}(\Psi_{\sigma}(u))}(\Omega_{R}(x_0), \varepsilon^{-a}\lambda) = 0 \ \mbox{ for all } \ \varepsilon^{-a} > 3^{n+1},$$ 
under assumption $\mathbf{M}_{\alpha}(\Psi_{\sigma}(u))(x_2) \le \lambda$. For this reason, one may conclude~\eqref{est:4.3} from~\eqref{est:4.3.2} by using the fact that
\begin{align}\label{est:4.3.3}
\mathbf{M}_{\alpha}^{R}(\Psi_{\sigma}(u))(z) = \sup_{0<r<R} r^{\alpha}\fint_{B_r(z)}\chi_{B_{2R}(x_0)}\Psi_{\sigma}(u)dx \le \mathbf{M}_{\alpha}(\chi_{B_{2R}(x_0)}\Psi_{\sigma}(u))(z).
\end{align} 
Here the last inequality of~\eqref{est:4.3.3} comes from a notice that $B_r(x) \subset B_{2R}(x_0)$ for any $r \in (0,R)$.
\end{proof}

\begin{lemma}\label{lem:A3}
For any $\alpha \in [0,\frac{n}{\theta})$ and $a > \frac{1}{\theta} - \frac{\alpha}{n}$, one can choose constants $b=b(\alpha,\theta,a,n)>0$ and $\varepsilon_0=\varepsilon_0(\alpha,\theta,a,b,\Upsilon,n,p)>0$ such that if there exist $x_1, \, x_2 \in B_R(x_0)$ satisfying
\begin{align}\label{est:cd 4.4}
\mathbf{M}_{\alpha}(\Psi_{\sigma}(u))(x_1) \le \lambda \hspace{0.2 cm} and \hspace{0.2 cm} \mathbf{M}_{\alpha}(|\mathbf{f}|^p + \Psi_{\sigma}(\mathsf{g}))(x_2) \le \varepsilon^b\lambda,
\end{align}
then the following inequality
\begin{align}\label{est:lem 4.4}
d_{\mathbf{M}_{\alpha}(\chi_{B_{2R}(x_0)}\Psi_{\sigma}(u))}(\Omega_R, \varepsilon^{-a}\lambda) \le \varepsilon \mathcal{L}^n(B_R(x_0)),
\end{align}
is valid for $\lambda>0$ and $0< \varepsilon <\varepsilon_0$.
\end{lemma}
\begin{proof}
We first consider the case when $x_0$ belongs to the interior domain of $\Omega$, that means $B_{4R}(x_0) \subset \Omega$. Suppose that $v$ solves the following homogeneous problem
\begin{equation}\nonumber
\begin{cases} -\mathrm{div} \mathbb{A}(x,\nabla v) + \mathbb{V} |v|^{p-2}v & = \ 0, \quad \ \quad \mbox{ in } B_{4R}(x_0),\\ 
\hspace{2.2cm} v & = \ u - \mathsf{g}, \ \mbox{ on } \partial B_{4R}(x_0).\end{cases}
\end{equation}
Applying the comparison estimate~\eqref{est:Comp} in Lemma~\ref{lem:Comp}, for all $\varepsilon_1  \in (0,1)$ one can choose $k = k(\theta,p)>0$ satisfying
\begin{align}\label{est:4.4.2}
\fint_{B_{4R}(x_0)} \Psi_{\sigma}(u-v) dx  \le \varepsilon_1  \fint_{B_{4R}(x_0)} \Psi_{\sigma}(u)dx + C\varepsilon_1^{-k} \fint_{B_{4R}(x_0)} (|\mathbf{f}|^p + \Psi_{\sigma}(\mathsf{g})) dx.
\end{align}
With $\mathbb{V} \in \mathbb{RH}^{\theta}$ for some $\theta \in [\frac{n}{p}, n)$, thanks to~\eqref{est:RH} in Lemma~\ref{lem:RH}, if $[\mathbb{A}]^{r_0} \le \delta$ for some $r_0>0$ and $\delta$ small enough then
\begin{align}\label{est:4.4.1}
\left(\fint_{B_{2R}(x_0)}(\Psi_{\sigma}(v))^{\theta} dx\right)^{\frac{1}{\theta}}\leq C\fint_{B_{4R}(x_0)} \Psi_{\sigma}(v) dx.
\end{align}
It is easy to find $c_p = C(p)>0$ satisfying
\begin{align*}
\Psi_{\sigma}(u) \le c_p \left(\Psi_{\sigma}(u-v) + \Psi_{\sigma}(v)\right),
\end{align*}
which allows us to find the following decomposition
\begin{align}\nonumber
I & := d_{\mathbf{M}_{\alpha}(\chi_{B_{2R}(x_0)}\Psi_{\sigma}(u))}(\Omega_R(x_0), \varepsilon^{-a}\lambda)  \le  C d_{\mathbf{M}_{\alpha}(\chi_{B_{2R}(x_0)}\Psi_{\sigma}(u-v))}(\Omega_R(x_0), c_p^{-1}\varepsilon^{-a}\lambda) \\ & \hspace{4cm} + Cd_{\mathbf{M}_{\alpha}(\chi_{B_{2R}(x_0)}(\Psi_{\sigma}(v)))}(\Omega_R(x_0), c_p^{-1}\varepsilon^{-a}\lambda). \label{est:4.4.3}
\end{align}
Let us apply Lemma~\ref{lem:M_alpha} to two last terms in~\eqref{est:4.4.3}, one arrives that
\begin{align*}
I \le C &\left(\frac{1}{\varepsilon^{-a}\lambda}\int_{B_{2R}(x_0)}\Psi_{\sigma}(u-v)dx\right)^{\frac{n}{n-\alpha}} + C\left(\frac{1}{(\varepsilon^{-a}\lambda)^\theta}\int_{B_{2R}(x_0)}(\Psi_{\sigma}(v))^\theta dx\right)^{\frac{n}{n-\alpha\theta}}, 
\end{align*}
which will be rewritten in the average form of integral as below
\begin{align}
I\le C&\left(\frac{R^n}{\varepsilon^{-a}\lambda}\fint_{B_{4R}(x_0)}\Psi_{\sigma}(u-v)dx\right)^{\frac{n}{n-\alpha}} + C\left(\frac{R^n}{(\varepsilon^{-a}\lambda)^\theta}\fint_{B_{2R}(x_0)}(\Psi_{\sigma}(v))^\theta dx\right)^{\frac{n}{n-\alpha\theta}}. \label{est:4.4.4}
\end{align}
Substituting~\eqref{est:4.4.2} and~\eqref{est:4.4.1} into~\eqref{est:4.4.4} to arrive
\begin{align}\nonumber
I & \le C\left(\frac{R^n}{\varepsilon^{-a}\lambda}\right)^{\frac{n}{n-\alpha}}\left(\varepsilon_1  \fint_{B_{4R}(x_0)} \Psi_{\sigma}(u)dx + C\varepsilon_1^{-k} \fint_{B_{4R}(x_0)} (|\mathbf{f}|^p + \Psi_{\sigma}(\mathsf{g})) dx \right)^\frac{n}{n-\alpha}\\ 
& \hspace{3cm} + C\left(\frac{R^{\frac{n}{\theta}}}{\varepsilon^{-a}\lambda}\fint_{B_{4R}(x_0)} \Psi_{\sigma}(v)dx \right)^{\frac{n\theta}{n-\alpha\theta}}. \label{est:4.4.5}
\end{align}
Since $x_1$ and $x_2 \in B_R(x_0)$, it can be seen that $B_{4R}(x_0) \subset B_{5R}(x_1) \cap B_{5R}(x_2)$, therefore one has 
\begin{align}\label{est:4.4.6}
\fint_{B_{4R}(x_0)} \Psi_{\sigma}(u)dx & \le C \fint_{B_{5R}(x_1)} \Psi_{\sigma}(u)dx \notag \\
& \le C R^{-\alpha}\mathbf{M}_{\alpha}(\Psi_{\sigma}(u))(x_1) \le CR^{-\alpha}\lambda, 
\end{align}
and
\begin{align}\label{est:4.4.7}
\fint_{B_{4R}(x_0)} (|\mathbf{f}|^p + \Psi_{\sigma}(\mathsf{g})) dx & \le C\fint_{B_{5R}(x_2)} (|\mathbf{f}|^p + \Psi_{\sigma}(\mathsf{g})) dx \notag \\
& \le C R^{-\alpha}\mathbf{M}_{\alpha}(|\mathbf{f}|^p + \Psi_{\sigma}(\mathsf{g}))(x_2) \le CR^{-\alpha}\varepsilon^{b}\lambda.
\end{align}
Moreover, using \eqref{est:4.4.2} again with notice that $0<\varepsilon_1 <1$, there holds
\begin{align} \nonumber
\fint_{B_{4R}(x_0)} \Psi_{\sigma}(v)dx &\le C\left(\fint_{B_{4R}(x_0)} \Psi_{\sigma}(u)dx + \fint_{B_{4R}(x_0)}\Psi_{\sigma}(u-v)dx\right) \\
& \le C\left(\fint_{B_{4R}(x_0)} \Psi_{\sigma}(u)dx + \varepsilon_1^{-k}\fint_{B_{4R}(x_0)} (|\mathbf{f}|^p + \Psi_{\sigma}(\mathsf{g})) dx\right). \label{est:4.4.8}
\end{align}
Collecting estimates in~\eqref{est:4.4.6}, \eqref{est:4.4.7} and~\eqref{est:4.4.8}, one obtains from~\eqref{est:4.4.5} that
\begin{align}\nonumber
I & \le C\left(\frac{R^n}{\varepsilon^{-a}\lambda}\right)^{\frac{n}{n-\alpha}}\left(\varepsilon_1 R^{-\alpha}\lambda + \varepsilon_1^{-k}R^{-\alpha}\varepsilon^{b}\lambda \right)^{\frac{n}{n-\alpha}}\\ \nonumber
& \hspace{2cm} + C\left(\frac{R^\frac{n}{\theta}}{\varepsilon^{-a}\lambda}\left(R^{-\alpha}\lambda + \varepsilon_1^{-k} R^{-\alpha}\varepsilon^b\lambda \right)\right)^{\frac{n\theta}{n-\alpha\theta}} \\ \label{est:4.4.9}
& = C\left[\varepsilon^{\frac{an}{n-\alpha}}\left(\varepsilon_1 + \varepsilon_1^{-k}\varepsilon^b\right)^{\frac{n}{n-\alpha}} + \varepsilon^{\frac{an\theta}{n-\alpha\theta}}\left(1+\varepsilon_1^{-k}\varepsilon^b\right)^{\frac{n\theta}{n-\alpha\theta}}\right]R^n.
\end{align}
In~\eqref{est:4.4.9}, we can choose $b>0$ and $\varepsilon_1 \in (0,1)$ such that
\begin{align}\label{eq:4.4}
b > \max\left\{0; \ \left(1-a-\frac{\alpha}{n}\right)(k+1)\right\} \ \mbox{ and } \ \varepsilon_1 = \varepsilon^{\frac{b}{1+k}}.
\end{align}
With this choice of parameters and the initial assumption $a > \frac{1}{\theta}-\frac{\alpha}{n}$, one can check that
\begin{align*}
\frac{an\theta}{n-\alpha\theta} > 1, \quad   \varepsilon^{\frac{an}{n-\alpha}} \varepsilon_1^{\frac{n}{n-\alpha}} = \varepsilon^{\left(a + \frac{b}{1+k}\right)\frac{n}{n-\alpha}} \ \mbox{ with } \ \left(a + \frac{b}{1+k}\right) \frac{n}{n-\alpha} >1,
\end{align*}
which allows us to conclude~\eqref{est:lem 4.4} from~\eqref{est:4.4.9} for every $\varepsilon$ small enough.\\

Let us now consider the remaining case when $x_0$ is near $\partial\Omega$, that means $B_{4R}(x_0) \cap \partial\Omega \neq \varnothing$. In this case, there is $x_3 \in \partial\Omega$ such that $|x_3 - x_0|$ = dist$(x_0, \partial\Omega) \le 4R$. We denote by $\tilde{v}$ the solution to 
\begin{equation}\nonumber
\begin{cases} -\mathrm{div}  \mathbb{A}(x,\nabla \tilde{v})  + \mathbb{V} |\tilde{v}|^{p-2}\tilde{v} & = \ 0, \quad \ \quad \mbox{ in } \Omega_{24R}(x_3),\\ 
\hspace{2.2cm} \tilde{v} & = \ u - \mathsf{g} , \ \ \mbox{ on } \partial \Omega_{24R}(x_3).\end{cases}
\end{equation}
It is similar to~\eqref{est:4.4.3}, since $B_{2R}(x_0) \subset B_{6R}(x_3)$, one can apply Lemma~\ref{lem:M_alpha} to write
\begin{align*}
I &\le d_{\mathbf{M}_{\alpha}(\chi_{B_{6R}(x_3)} \Psi_{\sigma}(u)}(\Omega_R(x_0), \varepsilon^{-a}\lambda)\\
\nonumber &\le  Cd_{\mathbf{M}_{\alpha}(\chi_{B_{6R}(x_3)}\Psi_{\sigma}(u-\tilde{v}))}(\Omega_R(x_0), c_p^{-1}\varepsilon^{-a}\lambda) +
Cd_{\mathbf{M}_{\alpha}(\chi_{B_{6R}(x_3)}\Psi_{\sigma}(\tilde{v}))}(\Omega_R(x_0), c_p^{-1}\varepsilon^{-a}\lambda)\\
&\le C\left(\frac{1}{\varepsilon^{-a}\lambda}\int_{B_{6R}(x_3)} \Psi_{\sigma}(u-\tilde{v}) dx\right)^{\frac{n}{n-\alpha}} + C\left(\frac{1}{(\varepsilon^{-a}\lambda)^\theta}\int_{B_{6R}(x_3)} (\Psi_{\sigma}(\tilde{v}))^\theta dx\right)^{\frac{n}{n-\alpha\theta}},
\end{align*}
which gives us
\begin{align}\label{est:4.4.11}
I \le  C\left(\frac{R^n}{\varepsilon^{-a}\lambda}\fint_{B_{6R}(x_3)} \Psi_{\sigma}(u-\tilde{v}) dx\right)^{\frac{n}{n-\alpha}} +C\left(\frac{R^n}{(\varepsilon^{-a}\lambda)^\theta}\fint_{B_{6R}(x_3)}(\Psi_{\sigma}(\tilde{v}))^\theta dx\right)^{\frac{n}{n-\alpha\theta}}. 
\end{align}
Next we will use a similar technique as in the first case in which the idea comes from the combination of the local reverse H{\"o}lder and the comparison estimate near the boundary. More precisely, Lemma~\ref{lem:Comp} and~\ref{lem:RH-boundary} give us two following estimates
\begin{align}\label{est:4.4.12}
\left(\fint_{B_{6R}(x_3)}(\Psi_{\sigma}(\tilde{v}))^{\theta} dx\right)^{\frac{1}{p\theta}}\leq C\left(\fint_{B_{24R}(x_3)} \Psi_{\sigma}(\tilde{v}) dx\right)^{\frac{1}{p}},
\end{align}
and 
\begin{align}\label{est:4.4.13}
\fint_{B_{24R}(x_3)} \Psi_{\sigma}(u-\tilde{v}) dx  \le \varepsilon_1  \fint_{B_{24R}(x_3)} \Psi_{\sigma}(u)dx + C\varepsilon_1^{-k} \fint_{B_{24R}(x_3)} (|\mathbf{f}|^p + \Psi_{\sigma}(\mathsf{g})) dx,
\end{align}
for all $\varepsilon_1 \in (0,1)$. Similarly to the first case, we may use assumptions \eqref{est:cd 4.4} to show that
$$\fint_{B_{24R}(x_3)} \Psi_{\sigma}(u)dx \le CR^{-\alpha}\lambda \hspace{0.2 cm} \text{and} \hspace{0.2 cm} \fint_{B_{24R}(x_3)} (|\mathbf{f}|^p + \Psi_{\sigma}(\mathsf{g})) dx \le CR^{-\alpha}\varepsilon^b\lambda,$$
which with~\eqref{est:4.4.12} and~\eqref{est:4.4.13} ensures that
\begin{align}\label{est:4.4.14}
\fint_{B_{24R}(x_3)} \Psi_{\sigma}(u-\tilde{v}) dx  \le C\left(\varepsilon_1 +\varepsilon_1^{-k}\varepsilon^b\right)R^{-\alpha}\lambda,
\end{align}
and 
\begin{align}\nonumber
\fint_{B_{6R}(x_3)}(\Psi_{\sigma}(\tilde{v}))^\theta dx & \le C \left(\fint_{B_{24R}(x_3)}\Psi_{\sigma}(u)dx + \fint_{B_{24R}(x_3)} \Psi_{\sigma}(u-\tilde{v}) dx\right)^\theta \\ \nonumber
& \le C\left[\left(1+\varepsilon_1 +\varepsilon_1^{-k}\varepsilon^b \right)R^{-\alpha}\lambda\right]^\theta \\
& \le C\left[\left(1+\varepsilon_1^{-k}\varepsilon^b\right)R^{-\alpha}\lambda\right]^\theta. \label{est:4.4.15}
\end{align}
Substituting~\eqref{est:4.4.14} and~\eqref{est:4.4.15} into~\eqref{est:4.4.11} and then choosing the same parameters as in~\eqref{eq:4.4}, one also obtains~\eqref{est:lem 4.4} to complete the proof.
\end{proof}

\section{Regularity result in Lorentz spaces}\label{sec:proofs}

With lemmas presented in the previous sections, we are now able to prove the main theorems of this article. Let us first recall the well-known covering lemma that originally due to Vitali \cite{Vitali08}, and stated as a substitution of Calder\'on-Zygmund-Krylov-Safonov decomposition, see \cite[Lemma 4.2]{CC1995} for detailed proof.

\begin{lemma}[Covering Lemma]\label{lem:Cover}
Let $\Omega$ be $(\delta,r_0)$-Reifenberg for $r_0$, $\delta>0$ and two measurable sets $\mathcal{V}\subset \mathcal{W} \subset \Omega$. Assume that $0< \varepsilon <1$ and  $0< r \le r_0$ satisfying
\begin{itemize}
\item[i)] $\mathcal{L}^n\left(\mathcal{V}\right) \le \varepsilon \mathcal{L}^n\left(B_{r}(0)\right)$;
\item[ii)] $\forall x \in \Omega$ and $\varrho \in (0,r]$, if $\mathcal{L}^n\left(\mathcal{V} \cap B_{\varrho}(x)\right) > \varepsilon \mathcal{L}^n\left(B_{\varrho}(x)\right)$ then $\Omega \cap B_{\varrho}(x)  \subset \mathcal{W}$. 
\end{itemize}
Then there is $C=C(n)>0$ satisfying $\mathcal{L}^n\left(\mathcal{V}\right)\leq C \varepsilon \mathcal{L}^n\left(\mathcal{W}\right)$.
\end{lemma} 

\begin{proof}[Proof of Theorem \ref{theo:main-A}]
One can rewrite the inequality~\eqref{est:A} as the form
\begin{align*}
\mathcal{L}^n (\{x \in \Omega: \ \mathbf{M}_{\alpha}(\Psi_{\sigma}(u)) > \varepsilon^{-a}\lambda\}) & \le C \varepsilon \mathcal{L}^n\left(\left\{x \in \Omega: \ \mathbf{M}_{\alpha}(\Psi_{\sigma}(u)) > \lambda\right\}\right) \\
& \hspace{0.5cm} + \mathcal{L}^n (\{x \in \Omega: \ \mathbf{M}_{\alpha}(|\mathbf{f}|^p + \Psi_{\sigma}(\mathsf{g})) > \varepsilon^{b}\lambda\}),
\end{align*}
which may be easily seen as a consequence of $\mathcal{L}^n(\mathcal{V}_{\lambda}^{\varepsilon}) \le C\varepsilon\mathcal{L}^n(\mathcal{W}_{\lambda})$, where two measurable sets $\mathcal{V}_{\lambda}^{\varepsilon}$ and $\mathcal{W}_{\lambda}$ of $\Omega$ given as follows
$$\mathcal{V}_{\lambda}^{\varepsilon} = \lbrace x \in \Omega: \ \mathbf{M}_{\alpha}(\Psi_{\sigma}(u))(x) > \varepsilon^{-a}\lambda, \ \mathbf{M}_{\alpha}(|\mathbf{f}|^p + \Psi_{\sigma}(\mathsf{g}))(x)\le \varepsilon^{b} \lambda \rbrace,$$
$$\mathcal{W}_{\lambda} = \lbrace x \in \Omega: \ \mathbf{M}_{\alpha}(\Psi_{\sigma}(u))(x) > \lambda) \rbrace).$$

The basic idea is to use Lemma~\ref{lem:Cover} for $\mathcal{V}_{\lambda}^{\varepsilon}$ and $\mathcal{W}_{\lambda}$ with $\lambda>0$ and $\varepsilon >0$ small enough. More precisely, we only need to show that one can choose $b = b(a,\theta,\alpha,n,p) >0$ and $\varepsilon_0 = \varepsilon_0(a,b,n) \in (0,1)$ such that two following statements
\begin{itemize}
\item[i)] $\mathcal{L}^n\left(\mathcal{V_{\lambda}^{\varepsilon}}\right) \le \varepsilon \mathcal{L}^n\left(B_{r_0}(0)\right)$;
\item[ii)] $\forall x \in \Omega$ and $0 < R \le r_0/24$, if $\mathcal{L}^n\left(\mathcal{V_{\lambda}^{\varepsilon}} \cap B_{R}(x)\right) > \varepsilon \mathcal{L}^n\left(B_{R}(x)\right)$ then $\Omega \cap B_{R}(x)  \subset \mathcal{W_{\lambda}}$,
\end{itemize}
hold for $\varepsilon \in (0,\varepsilon_0)$.\\

Without restriction of generality, one may suppose that $\mathcal{L}^n\left(\mathcal{V}_{\lambda}^{\varepsilon}\right) \neq 0$ which gives us a point $x_1 \in \Omega$ satisfying  $\mathbf{M}_{\alpha}(|\mathbf{f}|^p + \Psi_{\sigma}(\mathsf{g}))(x_1) \le \varepsilon^{b}\lambda$. Thanks to Lemma~\ref{lem:A1}, one has
$$\mathcal{L}^n(\mathcal{V}_{\lambda}^{\varepsilon}) \le d_{\mathbf{M}_{\alpha}(\Psi_{\sigma}(u))}(\Omega_R(x), \varepsilon^{-a} \lambda) \le  \varepsilon \mathcal{L}^n(B_r(0)),$$
which guarantees i) for $\varepsilon$ small enough and $b > \max\left\{0; \ 1-a-\frac{\alpha}{n}\right\}$.\\
 
In order to show ii), we will prove that for $ R \in (0, r_0/24]$ and $x$ in $\Omega$ if $\Omega \cap B_{R}(x) \cap \mathcal{W_{\lambda}}^c \neq \varnothing$ then $\mathcal{L}^n\left(\mathcal{V_{\lambda}^{\varepsilon}} \cap B_{R}(x)\right) \le \varepsilon \mathcal{L}^n\left(B_{R}(x)\right)$. Indeed, if we can find $x_2 \in \Omega \cap B_R(x) \cap \mathcal{W_{\lambda}}^c$ and $x_3 \in \mathcal{V_{\lambda}^{\varepsilon}} \cap B_R(x)$ then
\begin{align*}
\mathbf{M}_{\alpha}(\Psi_{\sigma}(u))(x_2) \le \lambda \hspace{0.2 cm} \text{and} \hspace{0.2 cm} \mathbf{M}_{\alpha}(|\mathbf{f}|^p + \Psi_{\sigma}(\mathsf{g}))(x_3) \le \varepsilon^{b}\lambda.
\end{align*}
Using this fact, two Lemmas~\ref{lem:A2} an~\ref{lem:A3} give us the existence of $\varepsilon_0>0$ and $b>0$ such that for $0< \varepsilon <\varepsilon_0$ there holds
$$\mathcal{L}^n(\mathcal{V_{\lambda}^{\varepsilon}} \cap B_R(x)) \le d_{\mathbf{M}_{\alpha}(\chi_{B_{2R}(x)}\Psi_{\sigma}(u))}(\Omega_{R}(x_0), \varepsilon^{-a}\lambda) \le \varepsilon \mathcal{L}^n(B_R(0)).$$
The proof is finally complete by applying Lemma~\ref{lem:Cover}.
\end{proof}

\begin{proof}[Proof of Theorem \ref{theo:main-B}]
For every $0<q<\frac{n\theta}{n-\alpha\theta}$, one can take $a$ satisfying
\begin{align*}
\frac{1}{\theta} - \frac{\alpha}{n} < a < \frac{1}{q}.  
\end{align*}
With this value of $a$, one can apply Lemma~\ref{theo:main-A} to find $\delta > 0$, $b>0$ and $\varepsilon_0 \in (0,1)$ such that under hypothesis $(\mathcal{H}_{\delta})$, then the following estimate holds
\begin{align}\label{est:5.3.1}
d_{\mathbf{M}_{\alpha}(\Psi_{\sigma}(u))}(\varepsilon^{-a}\lambda) \le C\varepsilon d_{\mathbf{M}_{\alpha}(\Psi_{\sigma}(u))}(\lambda) + d_{\mathbf{M}_{\alpha}(|\mathbf{f}|^p + \Psi_{\sigma}(\mathsf{g}))}(\varepsilon^{b}\lambda),
\end{align}
for all positive $\lambda$ and $\varepsilon \in (0,\varepsilon_0)$. Changing of variable in the integral of the quasi-norm in $L^{q,s}(\Omega)$ for $0< s < \infty$, one has
\begin{align*}
\| \mathbf{M}_{\alpha}(\Psi_{\sigma}(u))\|_{L^{q,s}(\Omega)}^s &= q\int_{0}^{\infty}\lambda^{s-1}( d_{\mathbf{M}_{\alpha}(\Psi_{\sigma}(u))}(\lambda))^{\frac{s}{q}} {d\lambda} \\
& = \varepsilon^{-as}q\int_{0}^{\infty}\lambda^{s-1} (d_{\mathbf{M}_{\alpha}(\Psi_{\sigma}(u))}(\varepsilon^{-a}\lambda))^{\frac{s}{q}} {d\lambda},
\end{align*}
which with~\eqref{est:5.3.1} to arrive
\begin{align}\nonumber
\|\mathbf{M}_{\alpha}(\Psi_{\sigma}(u))\|_{L^{q,s}(\Omega)}^s & \le C\varepsilon^{-as + \frac{s}{q}}q\int_{0}^{\infty}\lambda^{s-1}( d_{\mathbf{M}_{\alpha}(\Psi_{\sigma}(u))}(\lambda))^{\frac{s}{q}} {d\lambda} \\ \label{est:5.3.2}
& \hspace{2cm} + C\varepsilon^{-as}q\int_{0}^{\infty}\lambda^{s-1}( d_{\mathbf{M}_{\alpha}(|\mathbf{f}|^p + \Psi_{\sigma}(\mathsf{g}))}(\varepsilon^{b}\lambda))^{\frac{s}{q}} {d\lambda}.
\end{align}
Let us change of variable again in~\eqref{est:5.3.2} and combine with a basic inequality, one has
\begin{align}\label{est:5.3.3}
\|\mathbf{M}_{\alpha}(\Psi_{\sigma}(u))\|_{L^{q,s}(\Omega)} & \le C\varepsilon^{-a + \frac{1}{q}}\| \mathbf{M}_{\alpha}(\Psi_{\sigma}(u))\|_{L^{q,s}(\Omega)} \notag \\
& \hspace{2cm} + C\varepsilon^{-a-b}\|\mathbf{M}_{\alpha}(|\mathbf{f}|^p + \Psi_{\sigma}(\mathsf{g}))\|_{L^{q,s}(\Omega)}.
\end{align}
We note here the inequality~\eqref{est:5.3.3} still holds even for the case $s = \infty$. Moreover, with the choice of $a$ at the beginning of the proof it is possible to choose $\varepsilon$ in~\eqref{est:5.3.3} such that $C\varepsilon^{-a+\frac{1}{q}} \le \frac{1}{2}$, to obtain \eqref{est:B}. 
\end{proof}


\end{document}